\documentclass[letterpaper,10pt]{amsart}
\usepackage{amsmath,amssymb,amsxtra,amsthm, amstext, amscd,amsfonts,fancyhdr, hyperref,enumerate,array,mathrsfs}

\usepackage[all,cmtip]{xy}

\usepackage[OT2,T1]{fontenc}

\newcommand{\ZZ}{\mathbb{Z}}
\newcommand{\RR}{\mathbb{R}}
\newcommand{\QQ}{\mathbb{Q}}
\newcommand{\CC}{\mathbb{C}}

\newcommand{\PP}{\mathbb{P}}

\newcommand{\OO}{\mathcal{O}}

\newtheorem*{question*}{Question}

\newcommand{\sym}{\textnormal{Sym}}

\newcommand{\McL}{\mathcal{L}}

\newcommand{\McS}{\mathcal{S}}

\newcommand{\ra}{\rightarrow}

\newcommand{\Alb}{\textnormal{Alb}}

\newcommand{\Pic}{\textnormal{Pic}}
\newcommand{\Nef}{\textnormal{Nef}}
\newcommand{\PEff}{\textnormal{PEff}}

\newtheorem{theorem}{Theorem}[section]
\newtheorem*{theorem*}{Theorem}
\newtheorem*{conjecture*}{Conjecture}
\newtheorem{conjecture}{Conjecture}

\newtheorem{corollary}{Corollary}[theorem]
\newtheorem{lemma}[theorem]{Lemma}
\newtheorem*{lemma*}{Lemma}
\newtheorem{definition}[theorem]{Definition}
\newtheorem{proposition}[theorem]{Proposition}

\newtheorem*{example*}{Example}

\newtheorem*{remark*}{Remark}
\newtheorem{notation}{Notation}[theorem]

\title{Effective Eigendivisors and the Kawaguchi-Silverman Conjecture}

\author{Brett Nasserden}
\address{Department of Pure Mathematics \\
	University of Waterloo }
\email{bnasserd@uwaterloo.ca}

\begin{document}

\begin{abstract}
Let $f\colon X\ra X$ be a surjective endomorphism of a normal projective variety defined over a number field. The dynamics of $f$ may be studied through the dynamics of the linear action $f^*\colon \Pic(X)_\RR\ra \Pic(X)_\RR$, which are governed by the spectral theory of $f^*$. Let $\lambda_1(f)$ be the spectral radius of $f^*$. We study $\QQ$-divisors $D$ with $f^*D=\lambda_1(f) D$ and $\kappa(D)=0$ where $\kappa(D)$ is the Iitaka dimension of the divisor $D$. We analyze the base locus of such divisors and interpret the set of small eigenvalues in terms of the canonical heights of Jordan blocks described by Kawaguchi and Silverman. Finally we identify a linear algebraic condition on surjective morphisms that may be useful in proving instances of the Kawaguchi-Silverman conjecture.  
\end{abstract}

\maketitle 

\section*{Acknowledgements}

I would like to thank Yohsuke Matsuzawa and De-Qi Zhang for their patience and kindness when pointing out and explaining a critical error in a previous version of this article. The claims made in this updated article are much more modest then the previous version.

\section{Introduction}
Let $f\colon X\ra X$ be a surjective endomorphism of a  normal projective variety defined over $\bar{\QQ}$. Our goal in this article will be to understand both the arithmetic and geometric aspects of $f$ under iteration. Associated to $f$ are two basic numerical invariants that have been well studied, see for example (\cite{MR4080251},\cite{MR4068299},\cite{MR3189467},\cite{MR3871505},\cite{1802.07388}).  The \emph{dynamical degree} is defined as \begin{equation}
\lambda_1(f):=\lim_{n\ra \infty}((f^n)^*H \cdot H^{\dim X-1})^{\frac{1}{n}},\end{equation}
where $H$ is any ample divisor. We think of the dynamical degree as a \emph{geometric} measure of the complexity of $f$ under iteration. There is also the arithmetic degree of a point $P\in X(\bar{\QQ})$

\begin{equation}
\alpha_f(P):=\lim_{n\ra\infty}h_X^+(f^n(P))^\frac{1}{n},
\end{equation}
where $h_X$ is any ample height function on $X$. We think of the arithmetic degree as a measure of the \emph{arithmetic} complexity of the $f$-orbit of $P$. More precisely the arithmetic degree measures the rate of growth of the heights of the points $f^n(P)$. The Kawaguchi-Silverman conjecture predicts a certain type of ergodic relationship between these invariants.
\begin{conjecture}[\cite{MR4080251}]
Let $X$ be a normal projective variety defined over $\bar{\QQ}$ and let $f\colon X\ra X$ be a surjective endomorphism. Let $P\in X(\bar{\QQ})$ and suppose that the orbit $\OO_f(P):=\{P,f(P),f^2(P),...\}$ is Zariski dense in $X$. Then $\alpha_f(P)=\lambda_1(f)$. 
\end{conjecture}
In (\cite[Proposition 3.6]{MR4070310}) it was shown that the Kawaguchi-Silverman conjecture holds if one can find a non-trivial $\QQ$-divisor $D$ with $f^*D\sim_\QQ \lambda_1(f)D$ and $\kappa(D)>0$, where $\kappa(D)$ is the Iitaka dimension of $D$. Here we begin a study of what occurs when we allow $\kappa(D)=0$. We obtain some simple results on the invariance of such an eigendivisor.

\begin{theorem}
	Let $X$ be a normal projective variety and let $f\colon X\ra X$ be a surjective endomorphism. Let $D$ be a non-principal divisor with $f^*D\sim_\QQ\lambda D$ for some integral $\lambda>1$. Suppose that $\kappa(D)=0$. Then there is a integral multiple of $D^\prime$ of $D$ such that 
	
	\[f^{-1}(\textnormal{Bs}( D^\prime))=\textnormal{Bs}( D^\prime).\]
	
\end{theorem}

Let $f\colon X\ra X$ be a surjective endomorphism and $H$ an integral ample divisor on $X$. Set $V_H=\textnormal{span}_\QQ\{(f^n)^*H\}$. In (\cite{MR4068299}) it was shown that $V_H$ is finite dimensional and that there are \emph{canonical height functions} associated to the Jordan blocks of $f^*\colon V_H\ra V_H$. Let $\lambda_1,...,\lambda_l$  be the eigenvalues of $f^*\mid_{V_H}$ with $\lambda_1(f)=\mid \lambda_i \mid$. If  $\hat{h}_{\lambda_i}$ are the canonical height functions associated to a Jordan form of $f^*\mid_{V_H}$ then we have the following results.

\begin{theorem}
Let $G_{f,H}=\{P\in X(K): \hat{h}_{\lambda_i}(P)=0,\ 1\leq i\leq l\}.$ If $G_{f,H}$ is not Zariski dense then the Kawaguchi-Silverman conjecture holds for $f$.	
\end{theorem}

As a corollary we obtain 

\begin{corollary}
The Kawaguchi-Silverman conjecture is true for $f$ $\iff$  $G_{f,H}$ contains no dense orbit for $f$.	
	
\end{corollary}

Finally we study the consequences of the following condition.

\begin{definition}
	Let $f\colon X\ra X$ be a surjective endomorphism of a normal projective variety over a number field $K$ with $\lambda_1(f)>1$. Let $H$ be an integral eigendivisor and let $V_H$ be as we defined previously. We say $f^*\mid_{V_H}$ has a good eigenspace if $f^*\mid_{V_H}$ has the following properties,
	
	\begin{enumerate}
	\item  $\lambda$ is the unique eigenvalue of absolute value $\lambda$ of $f^*\mid_{V_H}$.
	\item The multiplicity of all $\lambda$ Jordan blocks of $f^*\mid_{V_H}$ are of multiplicity 1.
	\item If the Jordan blocks of $f^*\mid_{V_H}$ associated to $\lambda$ can be taken to be integral nef divisor classes $D_1,...,D_l$. 
	\item There is some $1\leq i\leq l$ such that $\kappa(D_i)\neq 0.$ 
	\end{enumerate}

\end{definition} 

We then illustrate how this condition may be used in a number of cases. It is our hope that in the future this condition may be verified for specific families of varieties.

\begin{theorem}
	Let $X$ be a normal projective variety defined over a number field $K$ and $f\colon X\ra X$ a surjective endomorphism with $\lambda_1(f)>1$ that admits a good eigenspace with respect to some integral ample divisor $H$. Then the Kawaguchi-Silverman conjecture holds for $f$ in the following cases. 
	\begin{enumerate}
		\item Let $X$ be a normal projective variety with finitely generated and rational nef cone. Suppose that for every nef divisor $D$ there is an integer $m\geq 1$ such that $mD$ is effective.
		\item 	Let $X$ be a normal $\QQ$-factorial projective variety defined over a number field $K$ with $\rho(X)=2,\ \kappa(X)<0, \Alb(X)=0$. Suppose in addition that $f\colon X\ra X$ be a surjective ramified endomorphism that is not int-amplified. 
		\item Let $X=\PP E$ where $Y$ is a smooth projective variety with Picard number $1$ and $E$ is a nef vector bundle on $Y$ with $H^0(Y,E)\neq 0$ such that $E$ is not ample and $\kappa(\PP E,-K_{\PP E})\geq 0$. 
		\item Let $X=\PP E$ where $E$ is a vector bundle on an elliptic curve $C$ defined over $K$.
	\end{enumerate}
	
 \end{theorem}

\section{Background}

Here we collect the results that we will need in this article and give some basic background information. Given a dominant rational map $f\colon X\dashrightarrow X$ one may attach various numerical measures for the complexity of this action. 

\begin{enumerate}
	\item The first dynamical degree: 
	\[\lambda_1(f):=\lim_{n\ra \infty}((f^n)^*H \cdot H^{\dim X-1})^{\frac{1}{n}}\]
	
	Here $H$ is an ample divisor on $X$. The limit can be shown to exist and be independent of the choice of $H$. 
	\item Given a point $P\in X(\bar{\QQ})$ such that $f^n(P)$ is defined for all $n$ we define the upper arithmetic degree
	
	\[\overline{\alpha_f}(P):=\limsup_{n\ra\infty}h_H^+(f^n(P))^\frac{1}{n}\]
	
	and the lower arithmetic degree
	
	\[\underline{\alpha_f}(P):=\liminf_{n\ra\infty}h_H^+(f^n(P))^\frac{1}{n}\]
	
	where $H$ is an ample divisor, $h_H$ is a choice of height function for $H$, and $h_H^+(P)=\max\{1,h_H(P)\}$. We think of the lower and upper arithmetic degrees as arithmetic measures of the complexity of $f$.
\end{enumerate} 

A related and interesting conjecture to the Kawaguchi-Silverman conjecture is the following.

\begin{conjecture}[{\cite[Conjecture 1.4]{2002.10976}}]\label{conj:sAND}
Let $X$ be a normal projective variety and $f\colon X\ra X$ a surjective endomorphism. Let
\[\McS(X,f,N)=\{P\in X(K):[K,\QQ]\leq N,\alpha_f(P)<\lambda_1(f)\}\]

Then  $\McS(X,f,N)$ is not Zariski dense in $X$.

\end{conjecture}

An easy argument makes it clear that the sAND conjecture implies the Kawaguchi-Silverman conjecture. Moreover, the sAND is potentially interesting for varieties $X$ with $\kappa(X)>0$ where the Kawaguchi-Silverman conjecture is not; if $\kappa(X)>0$ then we cannot have have dense orbits.

The Kawaguchi-Silverman conjecture is almost completely open when the rational map is not a morphism. For example the question is still open for $\PP^2$. By contrast when the map is a morphism then it is known that the limit $\lim_{n\ra\infty}h_H^+(f^n(P))^\frac{1}{n}$ exists. We call the limit the arithmetic degree, which is denoted $\lim_{n\ra\infty}h_H^+(f^n(P))^\frac{1}{n}:=\alpha_f(P)$. In the setting of morphisms the conjecture is known in a much wider variety of cases. For example the conjecture is known in the following cases.

\begin{enumerate}
	\item Varieties with Picard number 1. (\cite{MR4080251})
	\item Abelian varieties. (\cite{MR4068299,MR3614521})
    \item Smooth projective surfaces. (\cite{MR3871505})
	\item Rationally connected varieties admitting an int-amplified endomorphism. (\cite{1908.11537})
	\item Hyper-Kahler Varieties.(\cite{1802.07388})
\end{enumerate}
See (\cite{MR3871505}) for further comments on what is known. In the morphism case a powerful tool is the following connection to linear algebra. The morphism $f\colon X\ra X$ induces by pull back linear isomorphisms $f^*\colon\Pic(X)_\RR\ra\Pic(X)_\RR$ and $f^*\colon N^1(X)_\RR\ra N^1(X)_\RR$. One may profitably study the dynamics of $f$ through the dynamics of these actions. The connection is follows, \[\lambda_1(f)=\textnormal{Spectral Radius}(f^*\colon N^1(X)_\RR\ra N^1(X)_\RR)=\{\max\mid \lambda\mid:\lambda \textnormal{ an eigenvalue of } f^*\}\}\] If $\lambda>1$ is an eigenvalue of $f^*$ then one uses the divisor classes with $f^*D\sim_\QQ\lambda D$ to construct \emph{height functions}. The main results are as follows.

\begin{theorem}[{\cite[Theorem 5]{MR4080251}},{\cite[Theorem 1.1]{MR1255693}}]\label{thm:canheight1}
Let $X$ be a normal projective variety and $f\colon X\ra X$ a surjective endomorphism. Let $D\in \textnormal{Div}_\RR(X)$ and suppose that $f^*D\equiv_{\textnormal{Num}}\lambda D$ for some $\lambda>\sqrt{\lambda_1(f)}$. Then 
\begin{enumerate}
	\item For all $P\in \bar{\QQ}$ the limit $\hat{h}_D(P)=\lim_{n\ra\infty }\dfrac{h_D(f^n(P))}{\lambda^n}$ exists for any choice of height function $h_D$ associated to $D$.
	\item We have that $\hat{h}_D(f^n(P))=\lambda\hat{h}_D(P)$ and that $\hat{h}_D=h_D+O(\sqrt{h^+_X}\ )$ where $h_X$ is any ample height on $X$ and $h_X^+=\max\{1,h_X\}$. If $f^*D\sim_\QQ \lambda D$ then $\hat{h}_D=h_D+O(1).$
	\item If $\hat{h}_D(P)\neq 0$ then $\alpha_f(P)\geq \lambda$.
	\item If $\lambda=\lambda_1(f)$ and $\hat{h}_D(P)\neq 0$ then $\alpha_f(P)=\lambda_1(f)$. 
	\item If we are working over a number field and $D$ is ample then $\hat{h}_D(P)=0\iff P$ is pre-periodic for $f$.   
\end{enumerate}
\end{theorem}

We will need the following generalization to Jordan blocks, however in certain cases (\ref{thm:canheight1}) is sufficient and is the prototypical result.

\begin{theorem}[{\cite[Theorem 13]{MR4068299}}]\label{thm:canheight2}
Let $X$ be a normal projective variety over $\bar{\QQ}$ and let $f\colon X\ra X$ be a surjective endomorphism. Let $\lambda\in \CC$ with $\mid \lambda\mid >1$. Let $D_0,...,D_p\in \textnormal{Div}(X)_\CC$ with $f^*D_0\sim\lambda D_0$ and for $i\geq 1$ we have $f^*D_i\sim \lambda D_i+D_{i-1}$. We say that the $D_i$ are in Jordan block form. For each $D_i$ choose a Weil height $h_{D_i}$. Then we have the following.

\begin{enumerate}
	\item For each $i$ there are canonical height functions $\hat{h}_{D_i}\colon X(\CC)\ra \CC$ such that $\hat{h}_{D_i}=h_{D_i}+O(1)$ and $\hat{h}_{D_i}\circ f=\lambda \hat{h}_{D_i}+\hat{h}_{D_{i-1}}$ where we set $\hat{h}_{D_{-1}}=0$. 
	\item We have the following recursive formula.
	
	\[\hat{h}_{D_k}(x)=\lim_{n\ra\infty }(\lambda^{-n}h_{D_k}(f^n(x))-\sum_{i=1}^k\binom{n}{i}\lambda^{-i}\hat{h}_{k-i}(x)).\]
\end{enumerate} 
\end{theorem}

To use the above result we follow the ideas of Kawaguchi and Silverman in (\cite[Section 4]{MR4068299}). We will often use the following notation.

\begin{notation}\label{not:VH}
Choose an ample divisor $H\in\textnormal{Div}(X)$.

\begin{enumerate}
	\item We define $V_H$ to be the vector space spanned by $(f^n)^*H$ for all $n\geq 0$. By (\cite{MR4068299}) this is a finite dimensional space.
	\item Notice that by construction we have a linear mapping $f^*\colon V_H\ra V_H$. We will be interested in the eigenvalues of this linear mapping.
	\item Let $\lambda_1,..., \lambda_\sigma,\mu_{\sigma+1},...,\mu_d$ be the eigenvalues of $f^*\mid_{V_H}$ ordered such that 
	
	\[\mid \lambda_1\mid\geq \mid \lambda_2\mid\geq ...\geq \mid\lambda_\sigma\mid>1\geq \mid \mu_{\sigma+1}\mid\geq...\geq \mid \mu_d\mid.\]
We define $l=l_H$ to be the number of eigenvectors $\lambda_i$ such that $\mid\lambda_i\mid =\lambda_1(f)$. In particular we have that
\[\mid \lambda_i\mid =\lambda_1(f)\iff 1\leq i\leq l.\] 	
\item By possibly extending scalars we can find a Jordan form for $f^*$ on $V_H$ which means we find divisors $D_1,...,D_p$ with $f^*D_i\sim\lambda_i D_i$ or $f^*D_i\sim \lambda_i D_i+D_{i-1}$. We say the choice of divisors $D_i$ is a Jordan form for $f^*$.
	
\end{enumerate}     
\end{notation} 
With this notation we have the following key result.  

\begin{theorem}[{\cite[Section 4]{MR4068299}}]\label{theorem:KS2} With the above notation in place let $P\in X(\bar{\QQ})$. 

\begin{enumerate}
	\item Then $\alpha_f(P)=1$ or $\alpha_f(P)=\mid \lambda_i\mid$. More precisely suppose that $\hat{h}_{D_i}(P)\neq 0$ for some $1\leq i\leq \sigma$. Let $k$ be the smallest index with $\hat{h}_{D_k}(P)\neq 0$. Then $\alpha_f(P)=\mid \lambda_k\mid$. On the other hand if $\hat{h}_{D_k}(P)=0$ for all $1\leq k\leq \sigma$ then $\alpha_f(P)=1$. 
	\item In particular if $\mid \lambda_i\mid =\lambda_1(f)$ for $i=1,...,l$ and $\mid\lambda_{l+1}\mid <\lambda_1(f)$ then $\alpha_f(P)=\lambda_1(f)\iff \hat{h}_{D_i}(P)\neq 0$ for some $1\leq i\leq \sigma$  
\end{enumerate} 	
\end{theorem}

We will also heavily make use of the fact that the set of points in $X$ with $\alpha_f(P)=\lambda_1(f)$ is Zariski dense. 
\begin{theorem}[{\cite[Theorem 1.8]{2007.15180}}]\label{theorem:MSS} Let $X$ be a projective variety defined over $\bar{\QQ}$ and $f\colon X\ra X$ be a surjective endomorphism with $\lambda_1(f)>1$. Then the set of points $P\in X(\bar{\QQ})$ with $\alpha_f(P)=\lambda_1(f)$ is Zariski dense.
\end{theorem}	

We often will use the following fundamental theorem.

\begin{theorem}[{\cite[Corollary 27]{MR4080251}}{\cite[Theorem 1.4]{1606.00598}}]\label{thm:fundineq}
Let $X$ be a normal projective variety defined over a number field $K$ and $f\colon X\ra X$ a surjective endomorphism. If $P\in X(\bar{K})$ then $\alpha_f(P)\leq \lambda_1(f)$.

\end{theorem}

\begin{proposition}[Proposition 3.6, \cite{MR4070310}]\label{prop:yohsuke}
Let $X$ be a normal projective variety and $f\colon X\ra X$ a surjective endomorphism with $\lambda_1(f)>1$. Suppose that there is an non-trivial integral $\QQ$-cartier divisor $D$ with $f^*D=\lambda_1(f)D$ in $\Pic(X)_\QQ$. Then the Kawaguchi-Silverman conjecture holds for $f$.
		
\end{proposition}

Often we will attempt to reduce a question about the dynamical degree to a related variety by some sort of fibration. 
\begin{definition}[{\cite[Definition 2.1]{1802.07388}}]
Suppose that we have a commuting diagram of normal projective varieties defined over $\bar{\QQ}$

\[\xymatrix{X\ar[r]^f\ar[d]_\pi & X\ar[d]^\pi\\ Y\ar[r]_g & Y}\]

where $f,g,\pi$ are all surjective morphisms. The first relative dynamical degree of $f$ with respect to $\pi$ is defined to be \[\lambda_1(f\mid_\pi)=\lim_{n\ra\infty}((f^n)^*H_X\cdot (\pi^*H_Y)^{\dim Y}\cdot H_X^{\dim X-\dim Y-1})^{\frac{1}{n}}\]	

where $H_X,H_Y$ are ample divisors on $X$ and $Y$ respectively. 
\end{definition}

One may also look at (\cite{1802.07388}) for further references involving this notion, for example (\cite{MR2851870}). 

\begin{theorem}[{\cite[Theorem 2.2]{1802.07388}}]
Suppose that we have a commuting diagram of normal projective varieties defined over $\bar{\QQ}$

\[\xymatrix{X\ar[r]^f\ar[d]_\pi & X\ar[d]^\pi\\ Y\ar[r]_g & Y}\]

where $f,g,\pi$ are all surjective morphisms. Then 
\[\lambda_1(f)=\max\{\lambda_1(f\mid_\pi),\lambda_1(g)\}.\]
\end{theorem}

We obtain immediately 

\begin{corollary}[{\cite[Definition 2.7]{1802.07388}\label{cor:standard1}}]
Suppose that we have a commuting diagram of normal projective varieties defined over $\bar{\QQ}$

\[\xymatrix{X\ar[r]^f\ar[d]_\pi & X\ar[d]^\pi\\ Y\ar[r]_g & Y}\]

where $f,g,\pi$ are all surjective morphisms. Suppose that the Kawaguchi-Silverman conjecture holds for $g$ and $\lambda_1(f)=\lambda_1(g)$. Then the Kawaguchi-Silverman conjecture holds for $f$.	
\end{corollary}	
\begin{proof}
Let $P\in X(\bar{\QQ})$ be a point with dense $f$-orbit. Then $\pi(P)$ has a dense $g$ orbit and we have that $\alpha_f(P)\geq \alpha_g(\pi(P))=\lambda_1(g)=\lambda_1(f)$ by the Kawaguchi-Silverman conjecture for $g$ and our assumption on the dynamical degree. Since we know that $\alpha_f(P)\leq \lambda_1(f)$ the result follows. 
\end{proof}

The following result is crucial in the study of surjective morphisms of Mori-fiber spaces. 

\begin{lemma}[{\cite[Lemma 6.2]{1802.07388}}]\label{lemma:iterationlemma}
Let $\pi\colon X\ra Y$ be a mori fiber space. Suppose that $f\colon X\ra X$ is a surjective endomorphism. Then there is some iterate $f^n\colon X\ra X$ and $g\colon Y\ra Y$ such that

\[\xymatrix{X\ar[r]^{f^n}\ar[d]_\pi & X\ar[d]^\pi\\ Y\ar[r]_g & Y}\]

commutes. 

\end{lemma}
There is the following variant of the above.

\begin{lemma}[{\cite[Lemma 3.6]{1902.06071}}]
Let $X$ be a normal $\QQ$-factorial lc projective variety and $f\colon X\ra X$ a surjective endomorphism. Let $R$ is a $K_X$ negative extremal ray and $f_*R=R$. Let $\phi_R\colon X\ra Y$ be the associated extremal contraction. Then there is a morphism $g\colon Y\ra Y$ such that $g\circ \phi_R=\phi_R\circ f$.
\end{lemma}

We will also need something about the base locus and stable base locus of divisors. Our reference for this is (\cite{PosI}).

\begin{definition}\label{def:baseloci}
Let $X$ be a normal projective variety and $D$ a divisor on $X$.
	
	\begin{enumerate}
		\item We let $\textnormal{Bs}( D )$ be the \emph{base locus} of $D$. That is the set of points at which all sections of $H^0(X,D)$ vanish.
		\item We define the \emph{stable base locus} to be \[\textbf{B}(D)=\bigcap_{m\geq 1}\textnormal{Bs}( mD)\]
	\end{enumerate} 

\end{definition}

\section{Invariance of base locus}
Here we will give a simple argument that shows that the base locus of an integral eigendivisor is totally invariant.
\subsection{Eigendivisors with $\kappa(D)=0$}\label{subsec:heartofpaper} 
Let $X$ be a normal projective variety and suppose that $f\colon X\ra X$ is a surjective endomorphism. Suppose that we have a divisor $D$ with $D$ not linearly equivalent to $0$ with $f^*D\sim_\QQ\lambda D$ in the Picard group of $X$. When $\kappa(D)>0$ we have the following motivating result.

\begin{proposition}[{\cite[Proposition 3.5]{MR4070310}}]\label{prop:Lotsofsections}
Let $X$ be a normal projective variety defined over a number field $K$. Let $f\colon X\ra X$ be a surjective morphism defined over $K$. Let $D$ be a $\QQ$-divisor on $X$ with $f^*D\sim_\QQ\lambda_1(f) D$ with $\lambda_1(f)>1$ and $\kappa(D)>0$. Fix positive constants $A,B$. Then the set $\{P\in X(L):[L:K]\leq A,\hat{h}_D(P)\leq B\}$ is not Zariski dense in $X$. 	
\end{proposition}	
A natural weakening of (\ref{prop:Lotsofsections}) is to allow $\kappa(D)=0$. 
\begin{lemma}{\label{lemma:Baseinv}}
Let $X$ be a normal projective variety defined over $\bar{\QQ}$ and let $f\colon X\ra X$ be a finite surjective endomorphism. Take $D$ be a non-principal integral divisor with $f^*D\sim_\QQ\lambda D$ for some integral $\lambda>1$. Suppose that 

\begin{enumerate}
    \item $\boldsymbol{B}(D)=\textnormal{Bs}(mD)$ for all $m\geq 1$
	\item $H^0(X,mD)\cong H^0(X,D)$ for all $m\geq 1$.
\end{enumerate} 	

Then $f^{-1}(\boldsymbol{B}(D))=\boldsymbol{B}(D)$.	
\end{lemma}
\begin{proof}
First let $P\in \boldsymbol{B}(D)$. Let $s\in H^0(X,D)$. To show $f(P)\in \boldsymbol{B}(D)$ we must show that $s(f(P))=(f^*s)(P)=0$. Note that $f^*s\in H^0(X,\lambda D)$. By assumption we have that $\textnormal{Bs}(\lambda D)=\textnormal{Bs}(D)$ so $(f^*s)(P)=0$ for all $s\in H^0(X,D)$. Thus $f(P)\in \boldsymbol{B}(D)$ as needed. Conversely, suppose that $f(P)\in \boldsymbol{B}(D)$ for some $P$. This means that $s(f(P))=(f^*s)(P)=0$ for all $s\in H^0(X,D)$. On the other hand since $f$ is surjective and we are dealing with a vector bundle, $f^*$ is an injective vector space homomorphism. Thus by assumption

\[f^*\colon H^0(X,D)\ra H^0(X,\lambda D)\]

is an isomorphism. Thus for all $s^\prime\in H^0(X,\lambda D)$ we have $s^\prime=f^*s$ for some $s\in H^0(X,D)$. It follows that $P\in \textnormal{Bs}(\lambda D)=\boldsymbol{B}(D)$ as needed. 	
\end{proof}

\begin{proposition}
Let $X$ be a normal projective variety defined over $\bar{\QQ}$ and let $f\colon X\ra X$ be a surjective endomorphism. Let $D$ be a non-principal divisor with $f^*D\sim_\QQ\lambda D$ for some integral $\lambda>1$. Suppose that $\kappa(D)=0$. Then $f^{-1}(\boldsymbol{B}(D))=\boldsymbol{B}(D)$. 
	
\end{proposition}
\begin{proof}
	Since $f^*$ induces an injection on the group of sections 
	\[\dim H^0(X,D)=\dim f^*H^0(X,D)\leq\dim H^0(X,\lambda D)\]
	As $\kappa(D)=0$ there is some $m_2$ such that 
	
	\[\dim H^0(X,mD)=\dim H^0(X,m_2D)\]
	
	for all $m\geq m_2$. We may choose a large integer $m_1\geq m_2$ such that \[\textnormal{Bs}( km_1D)=\boldsymbol{B}(D)\] for all $k\geq 1$ by (\cite[2.1.21]{PosI}). In conclusion we have verified the needed hypothesis to apply (\ref{lemma:Baseinv})
	 
\end{proof}

\begin{corollary}\label{cor:invcor}
	Let $X$ be a normal projective variety and let $f\colon X\ra X$ be a surjective endomorphism. Let $D$ be a non-principal divisor with $f^*D\sim_\QQ\lambda D$ for some integral $\lambda>1$. Suppose that $\kappa(D)=0$. Then there is a integral multiple of $D^\prime$ of $D$ such that 
	
	\[f^{-1}(\textnormal{Bs}( D^\prime))=\textnormal{Bs}( D^\prime).\]
	
\end{corollary}

\subsection{Numerical vs Linear equivalence}

A basic problem to be overcome is the following. Let $f\colon X\ra X$ be a surjective endomorphism of a normal projective variety $X$. Let $\lambda_1(f)=\lambda>1$. Then $f^*$ acting on $N^1(X)_\RR$ has $\lambda$ as an eigenvalue of maximal absolute value. However, we would like to have $\lambda$ as an eigenvalue of largest absolute value when acting on $\Pic(X)_\RR$. This is almost true in the following sense.

\begin{proposition}\label{prop:eigenvaluesize}
Let $f\colon X\ra X$ be a surjective endomorphism of a normal projective variety $X$. Let $\lambda_1(f)=\lambda>1$. Then there is an eigenvalue $\lambda^\prime$ for $f^*$ acting on $\Pic(X)_\RR$ such that $\mid \lambda^\prime\mid=\lambda.$
\end{proposition}
\begin{proof}
Choose an ample divisor $H$ for $X$. Following (\ref{not:VH}) we have a finite dimensional vector space $V_H$. Let $E_1,...,E_p$be a Jordan form for $f^*$ after possibly extending scalars. By (\ref{theorem:MSS}) we can find a point $P$ such that $\alpha_f(P)=\lambda$. However by (\ref{theorem:KS2}) we have that $\lambda=\alpha_f(P)=\mid \lambda_i\mid.$ 
\end{proof}

\begin{corollary}\label{cor:findingeigenvalue1}
Let $f\colon X\ra X$ be a surjective endomorphism of a normal projective variety $X$. Let $\lambda_1(f)=\lambda>1$. Suppose that $\lambda$ is the unique eigenvalue of $f^*$ acting on $N^1(X)_\RR$ of largest absolute value. Then $\lambda$ appears as an eigenvalue of $f^*$ acting on $\Pic(X)_\RR$
\end{corollary}
\begin{proof}
By (\ref{prop:eigenvaluesize}) we have that there is an eigenvalue $\lambda^\prime$ of $f^*$ acting on $\Pic(X)_\RR$ of  absolute value $\lambda$. Thus $\lambda^\prime$ is an eigenvalue of absolute value $\lambda$ for the action of $f^*$ on $N^1(X)_\RR$.  However by assumption this means have $\lambda^\prime=\lambda$. 
\end{proof}

The above situation happens in practice.

\begin{theorem}[{\cite[Theorem 6.1]{MR4048444}}]
	Let $f\colon X\ra X$ be a surjective endomorphism of smooth projective varieties. Assume that $\lambda_1(f)^2>\lambda_2(f)$. Then $\lambda_1(f)$ is a simple eigenvalue of $f^*$ and is the only eigenvalue of modulus greater then $\sqrt{\lambda_2(f)}$. 
\end{theorem}

Here $\lambda_2(f)$ is a numerical invariant of $f$ related to how $f$ interacts with codimension 2 sub-varieties. So we should expect that it is often the case that there is a unique eigenvalue of largest absolute value and thus can find an eigendivisor $D$ for $\lambda_1(f)$ for linear equivalence.

\subsection{Analysis of Jordan Blocks}


Here we begin exploring our earlier results about eigendivisors with $\kappa(D)=0$ in the context of Kawaguchi-Silverman and the sAND conjectures. Our main tools will be (\cite[Section 4]{MR4068299}) and (\cite[Theorem 1.8]{2007.15180}) Let $f\colon X\ra X$ be a surjective endomorphism of a normal projective variety over a number field $K$. Suppose that $\lambda_1(f)=\lambda>1$. Choose an ample divisor $H$ of $X$ and follow (\ref{not:VH}). Let $d_H=d=\dim V_H$. After possibly extending scalars let $E_1,...,E_d$ be divisors such that the $E_i$ are a basis for $d$ and the associated matrix of $f^*$ acting on $V_H$ is in Jordan form. Since $\lambda>1$ recall that we have $l=l_H$ such that for $i\leq l$ we have $\mid \lambda_i\mid=\lambda $ and $\mid \lambda_{l+1}\mid<\lambda $.

\begin{definition}\label{def:altsmall}
Using (\ref{not:VH}) define

\[G_{f,H}=G=\{P\in X(K):\hat{h}_{\lambda_i}(P)=0\textnormal{ for }1\leq i\leq l\}\]
\end{definition}

This is the set of points of small height, relevant for the sAND conjecture,see (\ref{conj:sAND}).

\begin{proposition}\label{prop:altsAND}
Let $G_{f,H}$ be as in (\ref{def:altsmall}). Suppose that $G_{f,H}$ is not dense. Then the Kawaguchi-Silverman conjecture holds for $f$.  
\end{proposition}
\begin{proof}
 Let $P$ be a point with a dense $f$ orbit. I claim that $P\notin G$. Towards a contradiction suppose that $P\in G$. We show by induction that $f^n(P)\in G$ for all $n$ contradicting that $G$ is not dense in $X$. We have that $\hat{h}_{\lambda_1}(P)=0.$ This gives

\[\hat{h}_{\lambda_1}(f^n(P))=\lambda_1^n\hat{h}_{\lambda_1}(P)=0.\]

Thus for $i\leq l$ we may assume that for all $j<i$ and all $n$ we have that $\hat{h}_{\lambda_j}(f^n(P))=0.$ By assumption we have

\[\hat{h}_{\lambda_i}(P)=0\] and so for $n>1$ we have

\[\hat{h}_{\lambda_i}(f^n(P))=\lambda_i\hat{h}_{\lambda_i}(f^{n-1}(P))+\hat{h}_{\lambda_{i-1}}(f^{n-1}(P))=\lambda_i\hat{h}_{\lambda_i}(f^{n-1}(P))\]

by induction. Since we can repeat this process we have that
\[\hat{h}_{\lambda_i}(f^{n}(P))=0\] Thus if $P\in G$ then $\OO_f(P)\subseteq G$ contradicting that $G$ is not Zariski dense. Thus we have that if $\OO_f(P)$ is dense then $P\notin G$. Then for some $i\leq l$ we have that $\hat{h}_{\lambda_i}(P)\neq 0$. Then (\cite[Section 4]{MR4068299}) shows that $\alpha_f(P)=\mid\lambda_i\mid=\lambda$ as needed completing the proof.   
\end{proof}

The above lemma shows that the Kawaguchi-Silverman conjecture for morphisms is true provided one shows that the canonical Jordan block heights cannot cut out a Zariski dense set. We will use this principle to prove Kawaguchi-Silverman in some cases. We obtain the following rephrasing of the Kawaguchi-Silverman conjecture.

\begin{corollary}
The Kawaguchi-Silverman conjecture for an endomorphism $f$ with $\lambda=\lambda_1(f)>1$ is equivalent to $G_{f,H}$ contains no dense orbit of $f$. 
\end{corollary}
\begin{proof}
Suppose the Kawaguchi-Silverman conjecture holds for $f$. If $f$ has no dense forward orbit then trivially $G_{f,H}$ contains no forward orbit. Otherwise let $\OO_f(P)$ be dense. Then by Kawaguchi-Silverman $\alpha_f(P)=\lambda$. By (\cite[Section 4]{MR4068299}) we have that $\alpha_f(P)=\lambda\iff \hat{h}_{\lambda_i}(P)\neq 0$ for some $1\leq i\leq l$. So in particular, $P\notin G_{f,H}$ and so $\OO_f(P)$ is not contained in $G_{f,H}$. On the other hand suppose that $G_{f,H}$ contains no dense orbits. Let $\OO_f(P)$ be a dense orbit, by assumption $\OO_f(P)$ is not contained in $G_{f,H}$. Arguing as in (\ref{prop:altsAND}) we have that $P\notin G_{f,H}$ which means $\hat{h}_{\lambda_i}(P)\neq 0$ for some $1\leq i\leq l$ which means $\alpha_f(P)=\lambda$.
\end{proof}

So we may think of the Kawaguchi-Silverman conjecture as a statement about the structure of the set $G_{f,H}$. In fact, this set $G_{f,H}$ does not depend on $H$ and has been studied recently in (\cite{2002.10976}). Notice that $G_{f,H}=\{P\in X:\alpha_f(P)<\lambda\}$, which is the set of points of small arithmetic degree.  The following is a slight refinement of (\ref{theorem:MSS}).
\begin{proposition}\label{prop:densestatement}
	Let $f\colon X\ra X$ be an endomorphism of a normal projective variety over a number field $K$ with $\lambda_1(f)>1$. Let $H\in \textnormal{Div}(X)$ be ample and take $V_H$ as in (\ref{not:VH}). Let $E_1,...,E_\rho$ be a basis of $f^*\mid_{V_H}$ in Jordan block form. Define
	\[\mathcal{B}=\{P\in X(\bar{K}):\hat{h}_{E_i}(P)\neq 0\textnormal{ for some }1\leq i\leq l\}.\]
	
	Then $\mathcal{B}$ is dense in $X$. 	
\end{proposition}
\begin{proof}
Suppose that $\mathcal{B}\subseteq Y$ where $Y$ is closed. Then choose $P\notin Y$ with $\alpha_f(P)=\lambda_1(f)$ using (\ref{theorem:MSS}). Since $P\notin \mathcal{B}$ we have that

\[\hat{h}_{E_i}(P)=0\]	

for all $1\leq i\leq l$. If for some $l+1\leq i\leq \sigma$ we have $\hat{h}_{E_i}(P)\neq 0$ then by choosing $i$ minimal and applying (\cite[Section 4]{MR4068299}) we have

\[\alpha_f(P)=\mid \lambda_i\mid <\lambda\]

a contradiction. So we have that for all $i\leq \sigma$ we have $\hat{h}_{E_i}(P)=0$ which by (\cite[Section 4]{MR4068299}) gives $\alpha_f(P)=1$ which is again impossible. 
\end{proof}

\begin{proposition}
\label{prop:LargeJordanBlock}
Let $f\colon X\ra X$ be an endomorphism of a normal projective variety over a number field $K$ and use the notation of (\ref{prop:densestatement}). Suppose that for some $1\leq i\leq l$ we have that $E_i$ is an $\QQ$- divisor class with $\kappa(E_i)>0$. Then Kawaguchi-Silverman and the sAND conjecture holds for $f$. 
\end{proposition}
\begin{proof}
The sAND conjecture is the same as $G_{f,H}$ as defined in (\ref{def:altsmall}) not being dense. Notice that

\[G_{f,H}\subseteq \{P\in X(K): \hat{h}_{E_i}(P)=0\}\]

by the definition of $G_{f,H}$. By  (\cite[Proposition 3.5]{MR4070310}) we have that $\{P\in X(K): \hat{h}_{E_i}(P)=0\}$ is not Zariski dense and the result follows.

\end{proof}

\section{Some applications applications of eigendivisors with few sections}

\subsection{The case of a finitely generated nef cone}
	
Let $X$ be a normal projective variety defined over a number field $K$ with a finitely generated (not necessarily rational) nef cone, and Picard number $\rho(X)=\rho$. Suppose that we are given a $f\colon X\ra X$ is a surjective endomorphism. Suppose that the nef cone of $X$ has rays $v_1,...,v_s$. Since the action of $f^*$ is a linear isomorphism at the level of vector spaces that preserves the nef cone it preserves the boundary and so acts as a permutation on the rays. After iterating $f^*$ we may assume that $f^*$ fixes all of the rays. Since the nef cone is a full dimensional pointed cone, we have that $s\geq \rho$, and our assumption means that

\[f^*v_i=\lambda_i v_i.\]

In particular, by taking a linearly independent set of rays say $v_1,...,v_\rho$ we have that the action of $f^*$ on $N^1(X)_\RR$ is diagonalizable over $\RR$. Furthermore if $\mu$ is any eigenvalue of $f^*$ then $\mu=\lambda_i$ for some $i$ and the $\mu$ eigenspace has a basis 

\[\{v_j:\lambda_j=\mu\}.\]

To see why let $w$ be any $\mu$ eigenvector. Then we can write

\[w=\sum_{i=1}^{\rho}t_iv_i\textnormal{ and }\sum_{i=0}^{\rho}\mu t_iv_i=\mu w=f^*w=\sum_{i=1}^{\rho}\lambda_it_iv_i\]

So \[\mu t_i=\lambda_i t_i\]

Since since some $t_i\neq 0$ we see that $\lambda_i=\mu$ for some $i$ and that $t_j\neq 0\Rightarrow \lambda_j=\mu$. 

\begin{definition}
	Let $T\colon V\ra V$ be an invertible linear transformation of a finite dimensional real vector space that is diagonalizable. Let $C$ be a full dimensional pointed closed cone in $V$ with $T(C)=C$ and the rays of $V$ contain a basis of eigenvectors for $T$. Then given an eigenvalue $\lambda$ we say $C$ separates the $\lambda$-eigenspace if there exists $\lambda$-eigenvectors $v,w$ with $v$ and $w$ not lying on a common face.   
\end{definition}

The significance of the above definition follows from the following non-trivial result of convex geometry/spectral theory.

\begin{theorem}[{\cite[Theorem 4.8]{MR1832167}}]\label{thm:conethm}
	Let $A$ be a real $n\times n$ matrix. Then the following are equivalent.
	
	\begin{enumerate}
		\item $A$ is non-zero, diagonalizable, and all eigenvalues of $A$ have the same modulus, with $\rho(A)$ being an eigenvalue.
		\item There is a proper cone $K$ in $\RR^n$ with $A(K)\subseteq K$ and $A$ has an eigenvalue in the interior of $K$.
		
	\end{enumerate}
\end{theorem}

\begin{proposition}\label{prop:dialation}
	Let $X$ be a normal projective variety defined over a number field $K$ with a finitely generated (not necessarily rational) nef cone. Suppose that we are given a surjective endomorphism $f\colon X\ra X$. Suppose that $f^*$ preserves the rays of the nef cone $\Nef(X)_\RR$ and has positive eigenvalues. (This can always be achieved after iterating $f$) Then $\Nef(X)_\RR$ separates a $\lambda$-eigenspace for some eigenvalue $\lambda$ of $f^*$ if and only if $f^*$ is a dilation
\end{proposition}	
\begin{proof}
	If $f^*$ is a dilation by $\lambda$ then it certainly separates a $\lambda$-eigenspace. Now suppose that $f^*$ separates a $\lambda$-eigenspace for some eigenvalue $\lambda$. There are $\lambda$-eigenvectors $v,w$ that do not lie on the same face of $\Nef(X)_\RR$. Thus we have that $v+w$ is a $\lambda$-eigenvector on the interior of the nef cone. By (\ref{thm:conethm}) we have that all eigenvalues of $f^*$ have the same modulus. Since $f^*$ has real eigenvalues that are positive, all the eigenvalues coincide and $f$ is a dilation by $\lambda$ as needed. 
\end{proof}

We see that the obstruction to Kawaguchi-Silverman when the nef cone is finitely generated is that the eigenvectors of $f^*$ accumulate on a single facet of the nef cone. The philosophy is that as the nef cone of a variety gets more complicated, it becomes more and difficult for an endomorphism to preserve the nef cone unless the endomorphism is something like a dilation. We make this notion precise in the context of Kawaguchi-Silverman.

\begin{theorem}
Let $X$ be a normal projective variety defined over $\bar{\QQ}$. Suppose that $X$ has a finitely generated nef cone and $\rho=\rho(X)$ is the Picard number of $X$. Let $s$ be the number of rays of $\textnormal{Nef}(X)$. Let $f\colon X\ra X$ be a surjective endomorphism and let $t$ be the number of distinct eigenvalues of $f^*$ and let $q$ be the maximal number of rays in a facet of $\Nef(X)$. If $s\geq tq+1$ then some iterate $f^n$ has $(f^n)^*$ acting by a dilation and in particular Kawaguchi-Silverman and the sAND conjectures hold for $X$. In particular if $\rho=3$ then if $s\geq 3(3-1)+1=7$ the conclusion holds.
\end{theorem}
\begin{proof}
After iterating $f$ we may assume that $f^*$ fixes the rays of the nef cone and has positive eigenvalues. Let $\lambda$ be an eigenvalue of $f^*$ acting on $N^1(X)_\RR$. Towards a contradiction suppose that $f^*$ is not acting by a dilation. Suppose that an eigenvalue $\lambda$ has eigenvector $v_\lambda$ on a facet $F$ of $\Nef(X)$. If $\lambda$ appears as an eigenvalue of a ray on a different facet $F^\prime$ of $\Nef(X)$ we have that $\lambda$ appears as the eigenvalue of a ray in $F\cap F^\prime$, otherwise (\ref{prop:dialation}) implies that $f^*$ acts by a $\lambda$ dilation contradicting our assumptions. Thus there are at most $q$ rays that are an eigenvector for $\lambda$. Since there are $t$ eigenvalues we have that there are at most $tq$ rays, a contradiction. If $\rho=3$ then a facet is 2 dimensional cone, and thus has $2$ rays. So $q=2$ and $t\leq 3$. 
\end{proof}

\subsection{Varieties with a good eigenspace}\label{sec:effnefcone}

We recall the following situation which is the crux of the Kawaguchi-Silverman and sAND conjecture for varieties admitting an int-amplified endomorphism.  

\begin{definition}[{\cite{1908.01605}}]\label{def:CaseTIR}
	Let $X$ be an $n$-dimensional normal $\QQ$-factorial projective variety with at worst klt singularities that admits an int-amplified endomorphism. Let $f\colon X\ra X$ be a surjective endomorphism. The following situation is called case $\textnormal{TIR}_n$.  We assume the following conditions.
	
	\begin{enumerate}
		\item The anti canonical class $-K_X$ is nef and not big with $\kappa(X,-K_X)=0$.
		\item $f^*D=\lambda D$ in $\Pic(X)_\QQ$ for some effective irreducible $\QQ$-divisor $D$ with $D$ linearly equivalent to $-K_X$ and $\kappa(D)=0$. 
		\item The ramification divisor of $f$ is $\textnormal{Supp}(D)$.
		\item There is an $f$-invariant Mori fiber space 
		
		\[\xymatrix{X\ar[r]^f\ar[d]_\phi& X\ar[d]^\phi \\Y\ar[r]_g & Y}\]
		
		with $\lambda_1(f)>\lambda_1(g)$. 
		\item $\dim X\geq \dim Y+2\geq 3$. 
	\end{enumerate} 
\end{definition}

\begin{question*}[{\cite[Question 1.8]{1908.01605}}]
	Does case $\textnormal{TIR}_n$ ever occur?
\end{question*}

The interest in this technical case is that it is the remaining obstacle to performing an $f$ invariant minimal model program to prove the Kawaguchi-Silverman and sAND conjectures. Notice that in this case, it follows that the eigenspace of $\lambda_1(f)$ is 1-dimensional by the condition $\lambda_1(g)<\lambda_1(f)$. 

\begin{theorem}[Theorem 1.7,\cite{1908.01605}]
Let $X$ be a normal $\QQ$-factorial projective variety with at worst klt singularities that admits an int-amplified endomorphism. If the Kawaguchi-Silverman conjecture holds for the varieties appearing in case $\textnormal{TIR}_n$ (perhaps vacuously) then the Kawaguchi-Silverman conjecture holds for $X$.
\end{theorem}

We now give a variant of $\textnormal{TIR}_n$ that takes into account the possibility of a $\lambda_1(f)$ eigenspace of dimension greater then 1. 

\begin{definition}\label{def:Jordan Block}
	Let $f\colon X\ra X$ be a surjective endomorphism of a normal projective variety over a number field $K$ with $\lambda_1(f)>1$. Let $H$ be an integral eigendivisor and let $V_H$ be as is (\ref{not:VH}). We say $f^*\mid_{V_H}$ has a good eigenspace if $f^*\mid_{V_H}$ has the following properties,
	
	\begin{enumerate}
		\item  $\lambda$ is the unique eigenvalue of absolute value $\lambda$ of $f^*\mid_{V_H}$.
		\item The multiplicity of all $\lambda$ Jordan blocks of $f^*\mid_{V_H}$ are of multiplicity 1.
		\item If the Jordan blocks of $f^*\mid_{V_H}$ associated to $\lambda$ can be taken to be integral nef divisor classes $D_1,...,D_l$. 
		\item There is some $1\leq i\leq l$ such that $\kappa(D_i)\neq 0$. 
	\end{enumerate}

\end{definition} 

The benefit of this definition is that it is relatively simple, but has interesting consequences. It is our hope that the definition will motivate additional research into which varieties have surjective endomorphisms with a good eigenspace.

\begin{proposition}\label{thm:effnefcone}
Let $X$ be a normal projective variety with finitely generated rational nef cone. Suppose that for every nef divisor class $D$ in $\Pic(X)_\QQ$ there is an integer $m\geq 1$ such $mD$ is effective. If $f\colon X\ra X$ is a surjective endomorphism, and $H$ is an integral ample divisor such that $f^*
\mid_{V_H}$ has a good eigenspace. Then the Kawaguchi-Silverman conjecture holds for $f$.  
\end{proposition}
\begin{proof}
If $\lambda_1(f)=1$ the result follows, so assume $\lambda>1$. After iterating $f^*$ we may assume that $f^*$ fixes the rays of the nef cone. Furthermore we have that the eigenvalues of $f^*$ must be rational and thus integers since they are also algebraic integers. Let $D_1,...,D_l$ be a basis of eigendivisors associated to $\lambda_1(f)$ for the action of $f^*\mid_{V_H}$. As $f^*\mid_{V_H}$ has a good eigenspace we may take the $D_i$ nef and $\kappa(D_i)\geq 0$ by our assumption on $X$. Since $f^*\mid_{V_H}$ has a good eigenspace for some $j$ we have $\kappa(D_j)\neq 0$. As $\kappa(D_j)\geq 0$ we have $\kappa(D_j)>0$ and the result follows from (\ref{prop:yohsuke}). 
\end{proof}

The proof of (\ref{thm:effnefcone}) illustrates the basic idea behind the assumption of a good eigenspace. We attempt to put ourselves in a situation where we know $\kappa(D_i)\geq 0$ and then use the good eigenspace assumption to conclude that $\kappa(D)>0$. 

\begin{question*}
Let $X$ be a normal projective variety defined over a number field and $f\colon X\ra X$ a surjective endomorphism. 

\begin{itemize}
	\item What conditions on $X$ guarantee that there is an integral ample divisor $H$ such that $f^*\mid_{V_H}$ has a good eigenspace.
	\item Can one find interesting examples where $f$ has no good eigenspace for any ample integral $H$.
	\item If $X$ has a finitely generated rational nef cone, is it the case that $f$ has a good eigenspace for some ample integral $H$.
\end{itemize}
\end{question*}
We now show that one can get information about $\textnormal{TIR}_n$ from the assumption of having a good eigenspace.
This material may be well known and follows from (\cite[Remark 5.9 and 5.10]{1712.07533}) in the smooth case, but we provide arguments for completeness as we need to move beyond the smooth setting. 

\begin{proposition}\label{prop:TIRSize}
	Let $f\colon X\ra X$ be a surjective endomorphism with $X$ a normal projective variety defined over $\bar{\QQ}$ with surjective Albanese map. Let $f\colon X\ra X$ be a surjective endomorphism. Suppose that $f^*A\sim_\RR\lambda A$ for some non-zero $A\in \Pic^0(X)_\RR$. Then \[\lambda\leq \sqrt{\lambda_1(f)}\]
\end{proposition}
\begin{proof}
	Let $\pi\colon X\ra \Alb(X)$ be the projection. Recall that $\pi^*$ induces an isomorphism \[\pi^*\colon \Pic^0(\Alb(X))\ra \Pic^0(X)\] on the level of $\ZZ$-modules and so also as $\RR$-vector spaces. Thus we have that $A=\pi^*B$ for some $B\in \Pic^0(\Alb(X))$. We also have a commuting diagram
	
	\[\xymatrix{X\ar[r]^f\ar[d]_\pi & X\ar[d]^\pi \\ \Alb(X)\ar[r]_g & \Alb(X)}\]
	
	by the universal property of the Albanese variety. Note that the Albanese morphism is surjective. Thus we have that $f^*\pi^*B=\lambda \pi^* B=\pi^*g^* B$. Since $\pi^*$ is an isomorphism on $\Pic^0$ we have that $g^*B=\lambda B$. We have thus reduced to the case that $X$ is smooth (and in fact an abelian variety). By (\cite[Remarks 5.8 and 5.9]{1712.07533}) we have that \[\sqrt{\lambda_1(f)}\geq \sqrt{\lambda_1(g)}\geq \lambda\]
	
	as needed. 
	
\end{proof}

\begin{corollary}\label{cor:uniquelift}
	Let $f\colon X\ra X$ be a surjective endomorphism with $X$ a normal projective variety defined over $\bar{\QQ}$ with surjective Albanese map.  Suppose that $\mu$ is an eigenvalue of $f^*$ acting on $\Pic(X)_\RR$ with $\mid \mu\mid>\sqrt{\lambda_1(f)}$. Suppose that $D_1,D_2\in \Pic(X)_\RR$ are non-zero with $f^*D_i\sim_\RR\mu D_i$. If $D_1\equiv_{\textnormal{Num}}D_2$ then we have that $D_1\sim_\RR D_2$
\end{corollary}
\begin{proof}
	By assumption we have that $D_1=D_2+A_0$ for some $A_0\in \Pic^0(X)_\RR$. Then $A_0=D_1-D_2$. So $f^*A_0=\mu A_0$. Since $\mid \mu\mid>\sqrt{\lambda_1(f)}$ we have that $A_0=0$ by (\ref{prop:TIRSize}).	
\end{proof}	

We now obtain the following technical result that will be needed later.

\begin{lemma}\label{lem:DinVH}
	Let $f\colon X\ra X$ be a surjective endomorphism with $X$ a normal projective variety defined over $\bar{\QQ}$ with surjective Albanese map. Suppose that $\lambda_1(f)=\lambda>1$. Assume the $\lambda$-eigenspace of $f^*$ acting on $N^1(X)_\RR$ is 1 dimensional and that $\lambda$ appears with multiplicity 1 for the action of $f^*$ on $N^1(X)_\RR$. Suppose further that there is a unique eigenvalue of largest possible magnitude for the action on $N^1(X)_\RR$. Suppose that $D$ is a non-zero element of $\Pic(X)_\RR$ with $f^*D\sim_\RR\lambda D$. Let $H$ be any integral ample divisor on $X$ and let $V_H$ be as in (\ref{not:VH}). Then $D\in V_H$. 
\end{lemma}
\begin{proof}
	Let $D_1,..,D_p$ be a Jordan form for $f^*$ acting on $V_H$ after extending scalars. Then arguing as in $(\ref{cor:findingeigenvalue1})$ there is some eigenvalue $\lambda^\prime$ of $f^*$ on $V_H$ of magnitude $\lambda$. Since $\lambda^\prime$ is also an eigenvalue of $f^*$ acting on $N^1(X)_\RR$ we have that $\lambda^\prime=\lambda$ by assumption. Let $D^\prime$ be a $\lambda$-eigenvalue in $V_H$ which exists by the above argument. Since $\lambda$ is real we may take $D^\prime$ real and thus have that the numerical class of $D^\prime$ is a $\lambda$-eigenvector for $f^*$  on $N^1(X)_\RR$. Thus $aD^\prime\equiv_{\textnormal{Num}} D$ since both are eigenvectors for the action of $f^*$ on $N^1(X)_\RR$ and the eigenspace is 1 dimensional. Now (\ref{cor:uniquelift}) says that $aD^\prime=D$ and so $D\in V_H$ as needed. 
\end{proof}

The above result tells us that all the vector spaces in certain situations have "base points", which tells us that there are strong restrictions on the subspaces $V_H$.

\begin{proposition}\label{prop:TIRBlock}
	Let $f\colon X\ra X$ be a surjective endomorphism with $X$ a normal projective variety defined over $\bar{\QQ}$ with surjective Albanese map. Let $H$ be an integral ample divisor on $X$ and assume the notation of (\ref{not:VH}). Let $\lambda=\lambda_1(f)>1$.  Suppose that  $f^*$ acting on $N^1(X)_\RR$ has a 1 dimensional eigenspace and a size 1 Jordan block with a unique eigenvalue of largest size. Then $f^*$ acting on $V_H$ has a size 1 Jordan block and a 1 dimensional eigenspace for the unique largest eigenvalue. 
\end{proposition}
\begin{proof}
	After extending scalars to say $D_1,...,D_p$ assume that the action of $f^*$ on $V_H$ is in Jordan form. As in (\ref{lem:DinVH}) $\lambda$ appears as an eigenvalue, and no other eigenvalues have magnitude $\lambda$. Since we are dealing with a real eigenvalue, the Jordan blocks associated to $\lambda$ are real. Suppose that $D_1,...,D_l$ is such a Jordan block. I claim that $l=1$ and that all Jordan blocks have multiplicity 1. If not we have $f^*D_2=\lambda D_2+D_1$ in $\Pic(X)_\RR$ and that $D_2=aD_1+A_0$ for some $A_0\in \Pic(X)_
	\RR$ because there is a unique Jordan block of size 1 for the action of $f^*$ on $N^1(X)_\RR$. Thus $D_i\equiv_{\textnormal{Num}} aD_1$ for some non-zero scalar $a$. It follows that $f^*A_0=\lambda D_2+D_1-a\lambda D_1=\lambda A_0+D_1$. So $f^*A_0-\lambda A_0=D_1$ which means $D_1\in \Pic^0(X)_\RR$ and this is a contradiction. So all Jordan blocks have size 1. Now let $D_1,...,D_s$ be an eigen basis for $f^*$ acting on $V_H$. The same argument given earlier implies that if $D_i,D_j$ are $\lambda$-eigenvalues then they are numerically scalar multiples of one another, and so must be scalar multiples in $V_H$, a contradiction. It follows that $f^*$ acting on $V_H$ has 1-dimensional $\lambda$ eigenspace with Jordan blocks of size 1. 	
\end{proof}
\begin{lemma}\label{lem:TIRLem}
	Let $X,Y$ be  normal projective varieties defined over $\bar{\QQ}$. Let $f\colon X\ra X$ a surjective endomorphism and suppose that the Albanese morphism is surjective. Suppose that $\pi\colon X\ra Y$ is a surjective endomorphism, and that we have a commuting diagram 
	\[\xymatrix{X\ar[r]^f\ar[d]_\pi & X\ar[d]^\pi\\ Y\ar[r]_g & Y}\]
	We assume that $\rho(X)=\rho(Y)+1$ and that $\lambda_1(f)>\lambda_1(g)$. Suppose that there is an ample integral divisor $H$ with $f^*\mid_{V_H}$ having a good eigenspace. Then if $f^*D=\lambda_1(f) D$ in $\Pic(X)_\RR$ for some nef $\QQ$-Cartier divisor. Then $\lambda_1(f)$ is an integer and $\kappa(D)\neq 0$.
\end{lemma}
\begin{proof}
	Since we have assumed that $\lambda_1(g)<\lambda_1(f)=\lambda$ we have that $\lambda$ is not a eigenvalue of $g^*$ acting on $N^1(Y)_\RR$. Furthermore, there is no eigenvalue of magnitude $\lambda$ for $g^*$.  Then by the assumption on the Picard number we have that $\det f^*=\lambda\det g^*$. Since $\det f^*,\det g^*$ are integers we have that $\lambda$ is a rational number and thus an integer.Then we have that $D\in V_H$ by (\ref{lem:DinVH}). Since $\lambda$ has multiplicity 1 for the action of $f^*$ on $N^1(X)_\RR$ (\ref{prop:TIRBlock}) tells us the same holds for $f^*$ acting on $V_H$ and that this eigenspace is one dimensional. The assumption that $f$ has a good eigenspace now tells us that $\kappa(D)\neq 0$.
\end{proof}

Ideally one would then apply (\ref{lem:TIRLem}) to case $\textnormal{TIR}_n$, where in that case we have that $\kappa(D)\geq 0$ and the conclusion is then that $\kappa(D)>0$.

\subsection{The case of Picard number 2}\label{sec:Pic2}

We restrict to the case of a variety with Picard number $2$ and prove some basic results. In addition we will see how the requirement of a good eigenspace arises naturally. For completeness we prove some easy results in this case which may be well known.

\begin{proposition}
Let $X$ be a normal projective variety of Picard number $2$. Let $f\colon X\ra X$ be a surjective endomorphism. Suppose that the eigenspace of $\lambda_1(f)$ is 2-dimensional. Then the Kawaguchi-Silverman Conjecture holds for $f$.
\end{proposition}
\begin{proof}
Note that $f^*$ permutes the boundary of the nef cone of $X$. So in particular, the boundary must be eigendivisors. It follows that $f^*$ acts diagonally on $N^1(X)_\RR$. In particular $f^*$ has an ample eigendivisor and the result follows.
\end{proof}

In fact in the setting of Picard number 2 we may assume that the nef cone and Pseudo-effective cone coincide.

\begin{proposition}\label{prop:nef=peff}
Let $X$ be a normal projective variety of Picard number 2. Suppose that $\Nef(X)\neq\PEff(X)$. Let $f\colon X\ra X$ be a surjective endomorphism. Then $f^2$ acts by a dilation on $N^1(X)_\RR$. In particular Kawaguchi-Silverman and the sAND conjectures hold for $X$. 
\end{proposition}
\begin{proof}
Since $\rho(X)=2$ we may replace $f$ with $f^2$ and assume that $f^*$ fixes the rays of the nef cone and has positive eigenvalues. As we have assumed that the pseudo-effective cone $\PEff(X)$ is strictly larger then $X$, there is a boundary ray $D$ of $\Nef(X)$ that lies in the interior of $\PEff(X)$. Thus $f^*$ has an eigendivisor in the interior of $\PEff(X)$ a full-dimensional pointed cone. By (\ref{thm:conethm}) we see that $f^*$ acts by a dilation as needed.  \end{proof}

From the perspective of the minimal model program (\ref{prop:nef=peff}) suggests that it suffices to consider the varieties which arise at the final stage of the minimal model program. Let $X$ be a variety with terminal singularities defined over $\bar{\QQ}$ with $\rho(X)=2$. Suppose that $\Nef(X)=\PEff(X)$. 

\begin{enumerate}
	\item If $K_X$ is nef then there are no $K_X$ negative extremal contractions and the minimal model program tells us that $X$ is a minimal model.
	\item If $K_X$ is not nef then there is at least one $K_X$ negative extremal contraction. The assumption that $\Nef(X)=\PEff(X)$ tells us that this extremal contraction is of fibering type. Thus $X$ is a Mori-fiber space.  
\end{enumerate}

As we want to deal with integral eigendivisors we need to work with varieties whose nef cone contains a rational ray on the boundary. The following elementary computation shows that in the presence of a surjective endomorphism with distinct eigenvalues, the boundary is rational, or completely irrational. That is, we cannot have an irrational ray and a rational ray.

\begin{proposition}
	Let $X$ be a smooth projective variety of Picard number 2 defined over a number field $K$. Let the eigenvalues of $f^*$ acting on $N^1(X)_\QQ$ be $\lambda,\mu$ with $\lambda\neq \mu$. Suppose that $\mu$ is an integer and that $\mu$ has an integral eigendivisor in $N^1(X)_\QQ$. Then $\lambda$ is an integer and $\lambda$ has an integral divisor class in $N^1(X)_\QQ$.  
\end{proposition}
\begin{proof}
First note that

\[\lambda\mu=\deg f^*=d\in \ZZ\]
So in particular as $\mu$ is an integer, we have that $\lambda$ is rational. On the other hand, $\lambda$ is an algebraic integer, so it is an integer. Now choose an integral ample divisor $L$ on $X$. Let 
\begin{enumerate}
	\item $f^*L=aL+bH$
	\item Set $D=uL+vH$ and $f^*D=\lambda D$
\end{enumerate}
Then we have that \[f^*(uL+vH)=uaL+ub+\mu vH=uaL+(ub+\mu v)H=\lambda uL+\lambda vH\]

Putting this together gives and using that $u\neq 0$ by our assumptions gives

\[ua=\lambda u,ub+\mu v=\lambda v\]

As $f^*$ is defined over $\ZZ$ and $L,H$ are integral we have that $b$ is an integer. We have that the the $\lambda$ eigenspace is given by

\[uL+\frac{ub}{\lambda-\mu}H\] 

Taking $u=\lambda-\mu$ we obtain an integral $\lambda$ eigendivisor

\[(\lambda-\mu)L+bH\]
\end{proof}

As an aside, this tells us that we can gain some information about $X$ from the existence of an surjective endomorphism $f$.

\begin{corollary}
Let $X$ be a smooth projective variety defined over $\bar{\QQ}$ with Picard number 2. Suppose that $X$ admits surjective endomorphism that is not an automorphism. Then both rays of the nef cone are irrational or both rays are rational.
\end{corollary}

We have the following simple result.

\begin{proposition}\label{prop:smalleig1}
	Let $X$ be a normal projective variety defined over a number field $K$. Let $f\colon X\ra X$ be a surjective endomorphism with $\lambda_1(f)=\lambda>1$. Suppose $f^*D_\lambda=\lambda D_\lambda$ with $D_\lambda$ a non-trivial divisor class. Suppose $f^*D_\mu=\mu D_\mu$ where $0<\mu<1$ and $D_\lambda+D_\mu$ is ample. Then Kawaguchi-Silverman holds for $f$. 
\end{proposition}	
\begin{proof}
	By assumption $H=D_\lambda+D_\mu$ is ample. Now let $P$ be a point with a dense forward orbit. Then since $\mu<1$ we have that $-C<h_{D_\mu}(f^n(P))<C$ for some positive constant $C$ and all $n$. This follows from
	
	\[\mu h_{D_\mu}(P)-C^\prime<h_{D_\mu}(f(P))<\mu h_{D_\mu}(P)+C^\prime\]
	
	for some $C^\prime$ applied iteratively. Now assume for a contradiction that the canonical height function $\hat{h}_{D_\lambda}(P)=0$. Then we have that the ample height $h_H$ is bounded on $\OO_f(P)$ as
	
	\[h_H(f^n(P))=\hat{h}_{D_\lambda}(f^n(P))+h_{D_\mu}(f^n(P))=\lambda^n\hat{h}_{D_\lambda}(P)+h_{D_\mu}(f^n(P))<C \]
	
	Since $H$ is an ample height we have that $\OO_f(P)$ is a finite set contradicting that the orbit is dense. The result follows.
\end{proof}

\begin{proposition}\label{prop:nonintprop1}
	Let $X$ be a normal projective variety defined over a number field $K$. Suppose that $X$ is of Picard number 2. Let $f\colon X\ra X$ be a non-int amplified morphism with $\lambda_1(f)$ not an integer. Then Kawaguchi-Silverman holds for $f$. 
\end{proposition}
\begin{proof}
	After iterating $f$ we may assume that the eigenvalues of $f^*$ are positive. Suppose $\lambda=\lambda_1(f)>1$. If $f^*$ is assumed to be non-int amplified with eigenvalues $\lambda,\mu$ then $\mu\leq 1$. If $\mu=1$ then $\lambda$ is also an integer and the characteristic polynomial for $f$ splits over $\ZZ$. It follows that $f^*$ is diagonalizable over $\QQ$ and so $f^*$ would have integral eigendivisors. Thus we have that $\mu<1$. Let $f^*D_\lambda=\lambda D_\lambda$ and $f^*D_\mu=\mu D_\mu$ with $D_\lambda,D_\mu$ nef divisor classes with $D_\lambda+D_\mu$ ample. Then $H=D_\lambda+D_\mu$ is ample. By (\ref{prop:smalleig1}) the result follows.
\end{proof}

When $\kappa(X)<0$ as one has the ramification divisor to work with. We now turn to this case and argue by cases.
 Suppose now that $\kappa(X)<0$ and $\lambda_1(f)=\lambda>1$. First suppose that $f$ is etale. Then by adjunction $f^*K_X=K_X$. Since we are assuming that $\rho(X)=2$ we can assume now that $K_X$ or $-K_X$ is nef.  If $K_X$ is nef we have contradicted the abundance conjecture. So we assume that $K_X$ is not nef.

\begin{proposition}\label{prop:unramifiedendomorphism}
Let $X$ be a normal projective variety defined over $\bar{\QQ}$ that is $\QQ$-factorial with at worst terminal singularities of Picard number 2. Let $\kappa(X)<0$. 
Let $f\colon X\ra X$ be an unramified surjection. Assume that $K_X$ is not nef. (For example assume the abundance conjecture) then Kawaguchi Silverman holds for $f$.
\end{proposition}
\begin{proof}
We may assume that $\Nef(X)=\PEff(X)$ by (\ref{prop:nef=peff}). Since $f$ is unramified we have that the action of $f^*$ has two eigenvalues $1$ and $\lambda$ and we may assume that $\lambda>1$. Since $K_X$ is not nef we have that $-K_X$ is nef, and since $f^*(-K_X)=-K_X$ we may assume that $-K_X$ is not ample, as otherwise $f^*$ preserves an element in the interior of the big cone which by our assumptions is the ample cone which would imply that $\lambda_1(f)=1$ in this case. Thus we may assume that $-K_X$ is nef but not big. Let $D_\lambda$ be a numerical class corresponding to the $\lambda$-eigenvalue. Since $-K_X$ is nef we can contract one of the rays of the cone of curves via say $\phi\colon X\ra Y$. Since we cannot contract the ray generated  by $-K_X$ we have contracted the ray associated to $D_\lambda$.  Since we have assumed that $D_\lambda$ is not big, the contraction is of fibering type. After iterating $f$ we obtain a diagram

\[\xymatrix{X\ar[r]^f\ar[d]_\phi & X\ar[d]^\phi\\ Y\ar[r]_g & Y}\]

where $\phi\colon X\ra Y$ is a Mori fiber space. In this case $Y$ has Picard number 1 and $\phi^*H=D_\lambda$ for some ample $H\in N^1(Y)_\RR$. It follows that $\lambda_1(f)=\lambda_1(g)$ and so Kawaguchi-Silverman follows by (\ref{cor:standard1}). 
\end{proof}

We may now assume that $f$ is not etale. So $f^*K_X+R=K_X$ where $R$ is a non-zero effective $R$ divisor. We deal with the case $\Alb(X)=0$. In this example we see through a basic computation how the good eigenspace condition (and thus case $\textnormal{TIR}_n$) may arises naturally.

\begin{proposition}\label{prop:ramifiedpicardnum2}
	Let $X$ be a normal $\QQ$-factorial projective variety defined over a number field $K$ with at worst klt singularities and $\rho(X)=2,\ \kappa(X)<0, \Alb(X)=0$. Let $f\colon X\ra X$ be a surjective ramified endomorphism that is not int-amplified. In other words the eigenvalues of $f^*$ are $\lambda,\mu$ with $\mid \mu\mid \leq 1$. Suppose that $f$ has a good eigenspace. Then the Kawaguchi-Silverman conjecture is true for $f$. 
\end{proposition} 
\begin{proof}
 We may iterate $f$ and assume that $\lambda,\mu$ are positive and that the pseudoeffective and nef cones coincide. Suppose first that $K_X$ and $R$ are linearly independent. The matrix of $f^*$ with respect to the basis $K_X,R$ is then

\[\begin{bmatrix}1 & a\\-1 & b \end{bmatrix}\] 

That is we assume $f^*R=aK_X+bR$. By assumption we have that $\mu\leq 1$. If $\mu<1$ then (\ref{prop:smalleig1}) tells us that the Kawaguchi-Silverman conjecture is true for $f$. So we take $\mu=1$ and we may assume that $V_H=\Pic(X)_\QQ$ as otherwise $f$ has an ample eigendivisor. Then $b+1=\lambda+1$ and $b-a=\lambda$. So $b=\lambda$ and $a=0$. So in particular, the $\lambda$-eigenspace contains $R$. Since $R$ is also nef and the $\lambda$ eigenspace is 1-dimensional we have that $\kappa(R)>0$ by the assumption that $f$ has a good eigenspace. By (\ref{prop:yohsuke}) the Kawaguchi-Silverman conjecture is true for $f$. Now suppose that $R=aK_X$ for some $a$. Since $R$ is effective and $K_X$ is not we have that $a<0$ or that $b(-K_X)=R$ for some $b>0$. Then $-K_X$ is effective and $f^*K_X=K_X-R=K_X+bK_X=(b+1)K_X$. Thus we have that $f^*(-K_X)=(b+1)(-K_X)$. Since $b+1>1$ we have that $\lambda=b+1$. Therefore as above $R$ generates the $\lambda$-eigenspace and we have the same conclusion.
\end{proof}

\section{Projective Bundles}

The methods developed in this article have applications in the case of projective bundles. Given a variety $X$ with Picard number 1, projective bundles are one way to construct Picard number 2 varieties with a surjective morphism to $X$. They provide one systemic way of constructing potentially dynamically interesting surjective morphisms. The following result  gives another illustration of how one might reduce to the good eigenspace case.

\begin{theorem}\label{thm:bundlethm1}
Let $X$ be a smooth projective variety with Picard number $1$ and let $E$ be a nef vector bundle on $X$ with $H^0(X,E)\neq 0$ such that $E$ is not ample and $\kappa(\PP E,-K_{\PP E})\geq 0$. Suppose that there is integral ample divisor $H$ such that $f^*\mid_{V_H}$ has a good eigenspace. Then the Kawaguchi-Silverman conjecture holds for $X$. 
\end{theorem}
\begin{proof}
We may assume that the pseudo-effective and nef cones coincide, otherwise we may iterate $f$ and apply (\ref{thm:conethm}) to obtain that $f^*$ will act by a dilation. Put $\McL=\OO_{\PP E}(1)$. By assumption $\McL$ is nef, but not ample. Therefore $\McL$ generates a ray of $\Nef(\PP E)$ as $N^1(\PP E)_\RR$ is a 2-dimensional vector space. Let $f\colon \PP E\ra \PP E$ be a surjective endomorphism and put $\pi\colon \PP E\ra X$ for the bundle projection. After iterating $f$ we may also assume that there is some surjective $g\colon X\ra X$ with $\pi\circ f=g\circ \pi $.  We may assume that $\lambda=\lambda_1(f)>\lambda_1(g)\geq 1$ by (\ref{cor:standard1}). We have that $f^*\McL\equiv_{\textnormal{Num}}\lambda\McL$. If $K_X\neq 0$ then either $K_X$ is ample or $-K_X$ is ample. If $K_X$ is ample, then $X$ is of general type. It is well known that in this case $g$ is a finite order automorphism. Since a dense orbit for $f$ implies a dense orbit for $g$ we see that no point of $\PP E$ has a dense orbit and the Kawaguchi-Silverman conjecture for $X$ is trivially true. Otherwise we may assume that $-K_X$ is ample and so $X$ is Fano. In this case $\Alb(X)=0$. Then $\Alb(\PP E)=0$ and so $f^*\McL\sim_\QQ \lambda \McL$. Since $\kappa(\McL)\geq 0$ and $\McL$ spans the $\lambda$-eigenspace of $f^*$ by assumption we have that $\McL\in V_H$ as $V_H$ contains a $\lambda$ eigendivisor. By assumption we have that $f^*\mid_{V_H}$ has a good Jordan block and so $\kappa(\McL)\neq 0$. As $\kappa(\McL)\geq 0$ we have that $\kappa(\McL)>0$ and we obtain the Kawaguchi-Silverman conjecture by (\ref{prop:yohsuke}). So we may assume that $K_X=0$. As $\kappa(-K_{\PP E})\geq 0$ we see that $-K_{\PP E}$ is pseudo-effective. If $-K_{\PP E}$ is big then it is also ample as we have assumed the nef cone and pseudo-effective cone coincide. In a Fano variety any nef divisor has an effective multiple.  Since we have assumed that $f$ has a good eigenspace we may apply (\ref{thm:effnefcone}) to obtain the Kawaguchi-Silverman conjecture. Thus we may assume that $-K_{\PP E}$ generates a boundary ray of the pseudo-effective cone which we have assumed is the nef cone. Therefore $f^*(-K_{\PP E})\equiv_{\textnormal{Num}}\lambda(-K_{\PP E})$.  Suppose first that $\kappa(-K_{\PP E})>0$. Then by \cite[Prposition 1.6]{1908.01605} the Kawaguchi-Silverman conjecture holds for $f$. Assume that $\kappa(-K_{\PP E})=0$. We have that $f^*(-K_{\PP E})-R=-K_{\PP E}$ where $R$ is some effective divisor.  The adjunction formula tells us that
\[f^*(-K_{\PP E})=-K_{\PP E}+R\]
We now use the argument found in \cite[Proposition 9.2 (2)]{1908.01605}. By \cite[Proposition 6.1]{1908.01605} applied to $-K_{\PP E}$ we have that $f^{-1}(\textnormal{supp}(-K_{\PP E}))=\textnormal{supp}(-K_{\PP E})$. Write $-K_{\PP E}=\sum_{i}a_iD_i$ where the $D_i$ are prime divisors and the $a_i>0$. Then after iterating $f$ we have that $f^{-1}(\textnormal{supp}(D_i))=\textnormal{supp}(D_i)$. Since $D_i$ is a prime divisor this tells us that $f^*D_i=\mu_i D_i$ for some number $\mu_i$. On the other hand. Since $-K_{\PP E}$ is not the pull back of some line bundle on $X$ we have that $-K_{\PP E}\equiv_{\textnormal{Num}} r\McL$ where $r$ is the rank of $E$. Thus we have that $D_i\equiv_{\textnormal{Num}}b_i\McL $ as $\McL$ is extremal. Since we have that $f^*D_i\equiv_{\textnormal{Num}}\lambda_i D_i$ we have that each $\mu_i=\lambda$. In conclusion we may write $-K_{\PP E}=D$ where $D$ is some effective $\QQ$-Cartier divisor with $f^*D=\lambda D$ and $\kappa(D)\geq 0$. As we have assumed the existence of a good eigenspace we may apply (\ref{lem:TIRLem}) to obtain $\kappa(D)\neq 0$ and so $\kappa(D)>0$. Therefore by $(\ref{prop:yohsuke})$ the Kawaguchi-Silverman conjecture follows. 
\end{proof}

The case of projective bundles over a Fano variety with Picard number one was treated by different methods in (\cite{2003.01161}). We now specialize to the case of elliptic curves. One has the following result which motivates the rest of our work.

\begin{theorem}\cite[Corollary 6.8]{1802.07388}\label{thm:JohnMatt}
	Let $C$ be a smooth curve. Then the following are equivalent.
	\begin{enumerate}
		\item The Kawaguchi-Silverman conjecture holds for surjective endomorphisms of projective bundles $\PP E$ where $E$ is a vector bundle on $C$.
		\item The Kawaguchi-Silverman conjecture holds for surjective endomorphisms of projective bundles $\PP E$ where $E$ is a semi-stable degree 0 vector bundle on $C$.
	\end{enumerate}
\end{theorem}

Thus it suffices to consider semi-stable and degree zero vector bundles on $E$. We will need the results of (\cite{AtiyahBundles}) which we now review. 

\begin{theorem}[Atiyah \cite{AtiyahBundles}]\label{thm:AtiyahClass}
Let $C$ be an elliptic curve defined over $\bar{\QQ}$. 

\begin{enumerate}
	\item  For each $r\geq 1$ there is a unique indecomposable degree $0$ rank $r$ vector bundle on $C$ that has a non-zero global section. We call this vector bundle $F_r$. Furthermore, $\dim H^0(C, F_r)=1$ (\cite[Theorem 5]{AtiyahBundles}). 
	
	\item Every indecomposable degree 0 vector bundle of rank $r$ is of the from $F_r\otimes L$ for some unique degree zero line bundle $L$ (\cite[Theorem 5]{AtiyahBundles}).
	
	\item $F_r\otimes F_s=\bigoplus_{i} F_{r_i}$ for some $r_i$ (\cite[Lemma 18]{AtiyahBundles}).
	
	\item $\det F_r=\OO_X$ (\cite[Theorem 5]{AtiyahBundles}).
	
	\item Let $L$ be a line bundle of degree $0$. Then $F_r\otimes F_s
	\otimes L$ has a global section if and only if $L=\OO_X$ (\cite[Lemma 17]{AtiyahBundles}).  
	
	\item $F_r$ is self dual (\cite[Corollary 1]{AtiyahBundles}).
	\item $F_r=\sym ^{r-1}F_2$ (\cite[Theorem 9]{AtiyahBundles}).
\end{enumerate}		
\end{theorem}

\begin{proposition}\label{prop:symprop}
	Let $C$ be an elliptic curve defined over $\bar{\QQ}$ and let $F_r$ be the rank $r$ Atiyah bundle on $C$. Then for any $d\geq 1$ we have that \[\sym^d F_r=\bigoplus_{r_i}F_{r_i}\]

	for some integers $r_i$.
\end{proposition}
\begin{proof}
	Recall that we have a canonical surjection 
	\[\psi\colon F_r^{\otimes d}\ra \sym^d F_r\ra 0 \]
	
	By (\ref{thm:AtiyahClass}) we have that $ F_r^{\otimes d}=\bigoplus_i F_{l_i}$ for some integers $l_i$. Because $\deg F_r=0$ we have  $\deg \sym^d F_r=0$. By taking the dual endomorphism we have a mapping
	
	\[
	0\ra (\sym^d F_r)^*\ra (E^{\otimes d})\]
	
	Since we are in characteristic 0 and $F_r^*=F_r$ we have that $(\sym^d F_r)^*=\sym^d F_r^*=\sym^d F_r$ and $(F_r^{\otimes d})^*=F_r^{\otimes d}$. Thus $\sym^d F_r$ is a sub-bundle of $F_r^{\otimes d}$. Since $F_r^{\otimes d}$ is semi-stable of degree zero (being the tensor product of semi-stable vector bundles) we have that $\sym^d F_r$ is also semi-stable and degree $0$. Now let $W$ be an indecomposable component of $\sym^dF_r$. Since $\deg \sym^d F_r=0$ we have that $\deg W\leq 0$ as otherwise $W$ destabilizes $\sym^d F_r$. Now suppose that $\deg W<0$. Then since $\deg \sym^d F_r=0$ there is some other indecomposable component $W^\prime$ with $\deg W^\prime >0$. This contradicts the argument just given. So $\deg W=0$. Thus each component of $\sym^d F_r$ is degree $0$. It follows that each component is also semi-stable, being a degree zero sub-bundle of a semi-stable bundle of degree degree zero. From (\ref{thm:AtiyahClass}) we have that 
	
	\[\sym^d F_r=\bigoplus_{j} F_{r_j}
	\otimes L_j\]
	
	where $L_j$ is some degree 0 line bundle. Now we have that 
	
	\[\hom(F_r^{\otimes d},\sym^d F_r)=\hom(\oplus_{i}F_{l_i},\oplus_j F_{r_j}\otimes L_j)=\oplus_{i,j}\hom(F_{l_i},F_{r_j}\otimes L_j)\]
	
	Now we have that $\hom(F_{l_i},F_{r_j}\otimes L_j)=H^0(C,F_{l_i}^*\otimes F_{r_j}\otimes L_j)=H^0(C,(F_{l_i}\otimes F_{r_j}\otimes L_j)$ as the Atiyah bundle is self dual. Suppose that $L_j\neq \OO_C$ for some $j$. Then as $L_j$ is of degree $0$ we have that $H^0(F_{l_i}\otimes F_{r_j}\otimes L_j)=0$ by (\ref{thm:AtiyahClass}). However, this is a contradiction as $\psi$ is a surjection. So we have that each $L_j=\OO_C$ and the claim follows.
	
\end{proof}

\begin{proposition}\label{prop:antiiitaka}
Let $E$ be a semi-stable degree zero vector bundle on an elliptic curve $C$. \begin{enumerate}
	\item We may write $E=\bigoplus_{k=0}^NF_{r_k}\otimes L_k$ where $F_{r_k}$ are the rank $r_k$ Atiyah bundles and $L_k$ is a degree zero vector bundle.
	\item If $E=\bigoplus_{k=0}^NF_{r_k}\otimes L_k$ with $L_i$ degree zero line bundles with $L_0=0$ then $\kappa(-K_{\PP E})\geq 0$. 
\end{enumerate} 
\end{proposition}
\begin{proof}
As $E$ is semi-stable of degree 0 every sub-sheaf has degree at most $0$. Let $W_0,...,W_N$ be the components of $E$. If $\deg W_i<0$ as $\sum_{k=0}^N \deg W_i=0$ there must be some $W_i$ with $\deg W_i>0$ which would destabilize $E$. Thus every indecomposable component is degree zero. It follows from the definition of semi-stable that each irreducible component is also semi-stable. By (\ref{thm:AtiyahClass}) we have that $W_i=F_{r_i}\otimes L_i$ for some degree 0 line bundle as needed. We now turn to part 2. We have from the Euler exact sequence that

\[-K_{\PP E}=\OO_{\PP E}(r)-\det E\]

where $r$ is the rank of $E$. Then

\begin{align*}
H^0(\PP E,-K_{\PP E})&=H^0(\PP E, \OO_{\PP E}(r)\otimes \det E^{-1})\\
&=H^0(C,\pi_*(\OO_{\PP E}(r))\otimes \det H^{-1})=H^0(C,\sym^r (E)\otimes \det E^{-1})
\end{align*}

We have

\[\sym^r(E)=\bigoplus_{i_0+...+i_N=r}\bigotimes_{k=0}^N\sym^{i_k}(F_{r_k})\otimes L_k^{\otimes i_k}\]

from our assumed decomposition of $E=\bigoplus_{k=0}^NF_{r_k}\otimes L_k$. Therefore \[\sym^r(E)\otimes\det E^{-1}=\bigoplus_{i_0+...+i_N=r}\bigotimes_{k=0}^N\sym^{i_k}(F_{r_k})\otimes L_k^{\otimes i_k}\otimes \det E^{-1}\]

Let $j_0=r-N$ and $j_k=1$ for $k=1,...,N$. Then $\sum_{k=0}^Nj_k=r-N+N=r$. Therefore

\[\bigotimes_{k=0}^N\sym^{j_k}(F_{r_k})\otimes L_k^{\otimes j_k}\otimes \det E^{-1}\]

is a direct summand of $\sym^r(E)\otimes\det E^{-1}$. Since $\det E=\sum_{k=0}^N L_k$ and $L_0=\OO_C$ we have that 

\[\bigotimes_{k=0}^N\sym^{j_k}(F_{r_k})\otimes L_k^{\otimes j_k}\otimes \det E^{-1}=\bigotimes_{k=0}^N\sym^{j_k}(F_{r_k})\]

since $\sum_{k=0}^Nj_kL_k=\sum_{k=1}^Nj_kL_k=\sum_{k=1}^NL_k$ using $L_0=\OO_C$ and the construction of the $j_k$. We have shown that $\bigotimes_{k=0}^N\sym^{j_k}(F_{r_k})$ (which has a global section by (\ref{prop:symprop})) is a direct summand of $\sym^r(E)\otimes\det E^{-1}$. Therefore $-K_{\PP E}$ has a non-zero global section and $\kappa(-K_{\PP E})\geq 0$ as required.

\end{proof}

\begin{corollary}\label{cor:intamlike1}
Let $C$ be an elliptic curve defined over a number field. Let $E$ be a vector bundle on $C$. Then the Kawaguchi-Silverman conjecture holds for surjective endomorphisms of $\PP E$ that admit a good eigenspace.
\end{corollary}
\begin{proof}
By (\ref{thm:JohnMatt}) we may assume that $E$ is semi-stable fo degree zero. Using (\ref{prop:antiiitaka}) we may write $E=\bigoplus_{k=0}^{N}F_{r_k}\otimes L_K$ where the $L_k$ are degree zero line bundles. After twisting by $L_0^{\otimes -1}$ we may assume that $L_0=\OO_C$. Therefore $\kappa(-\PP_{E})\geq 0$ by (\ref{prop:antiiitaka}). By (\ref{thm:bundlethm1}) the Kawaguchi-Silverman conjecture is true for $f$ as we have assumed the existence of a good eigenspace.
\end{proof}

\bibliography{bib}{}
\bibliographystyle{plain}

\end{document}